\documentclass[11pt,reqno]{amsart}
\usepackage{amsmath}
\usepackage{amssymb}
\usepackage{amsthm}
\usepackage{enumerate}
\usepackage[usenames]{color}
\usepackage{eepic,epic}
\usepackage{url} 

\usepackage{makecell}
\usepackage{epstopdf}

\usepackage{graphicx}

\newtheorem{thm}{Theorem}[section]

\newtheorem{lem}[thm]{Lemma}
\newtheorem{prop}[thm]{Proposition}

\theoremstyle{definition}

\theoremstyle{remark}

\numberwithin{equation}{section}

\newtheorem{ass}{Assumption}

\begin{document}
\title[]
{On the convergence  of the distributed Proximal point algorithm}

\author{Woocheol Choi}
\address{}
\email{choiwc@skku.edu}



\maketitle

\begin{abstract} 
In this work, we establish convergence results for the distributed proximal point algorithm (DPPA) for distributed optimization problems. We consider the problem on the whole domain $\mathbb{R}^d$ and find a general condition 
on the stepsize and cost functions such that the DPPA is stable. We prove that the DPPA with stepsize $\eta>0$ exponentially converges to an $O(\eta)$-neighborhood of the optimizer.
Our result clearly explains the advantage of the DPPA with respect to the convergence and stability in comparison with the distributed gradient descent algorithm. We also provide numerical tests supporting the theoretical results.
\end{abstract}

\section{Introduction}

This works consider the distributed optimization 
\begin{equation}\label{goal}
\min_{x \in \mathbb{R}^d} f(x) = \sum_{i=1}^n f_i (x),
\end{equation}
where $n$ denotes the number of the agents in the network, and $f_i: \mathbb{R}^d \rightarrow \mathbb{R}$ is a differentiable local cost only known agent $i$ for each $1 \leq i \leq n$. 
In recent years, extensive research has been conducted on distributed optimization due to its relevance in various applications, including wireless sensor networks \cite{BG, SGRR}, multi-agent control \cite{MC1, MC2, MC3}, smart grids \cite{SG1, SG2}, and machine learning \cite{ML1, ML2, ML3, ML4, ML5}.

Numerous studies have focused on solving distributed optimization problems on networks. Relevant research works include \cite{NO1, NOS, SLWY, SLYWY}, and the references therein. A key algorithm in this field is the distributed gradient descent algorithm (DGD), introduced in \cite{NO1}. Additionally, there exist various versions of distributed optimization algorithms such as EXTRA \cite{SLWY} and decentralized gradient tracking \cite{NOS, QL, XK2}. Lately, there has been growing interest in designing communication-efficient algorithms for distributed optimization \cite{S, CO, JXM, BBKW}.

The distributed proximal point algorithm (DPPA) was proposed in \cite{MFGP2, MFGP} as a distributed counterpart of the proximal point method, analogous to the relation between the distributed gradient descent algorithm and the gradient descent method. The works  \cite{MFGP2, MFGP} established the asymptotic convergence of the DPPA under the assumptions of a compact domain for each local cost $f_i$ and a decreasing step size. The work \cite{LFX2} designed the DPPA on a directed graph and proved a convergence estimate with rate $O(1/\sqrt{t})$ when the step size is set to $1/\sqrt{t}$ and each cost function has a compact domain.

It is a well-known fact that the proximal point method is more stable compared to the gradient descent method for large choices of the stepsize. This fact suggests that the DPPA may also exhibit be more stable than the DGD, as mentioned in the previous works \cite{LFX, LFX2, MFGP}. In this work, we provide convergence results for the DPPA in the case of the entire domain and a constant step size. Comparing the results with the convergence result \cite{YLY} of the DGD, we will find that the DPPA is more stable than the DGD when the stepsize is large.

The DPPA for the problem \eqref{goal} is described as follows: 
\begin{equation}\label{eq-1-10}
\begin{split}
\hat{x}_{i}(t)& = \sum_{j=1}^n w_{ij} x_j (t)
\\
x_i (t+1)& = \textrm{argmin}_{x \in \mathbb{R}^n}\Big( f_i (x) + \frac{1}{2\eta} \|x - \hat{x}_i (t)\|^2 \Big),
\end{split}
\end{equation}
where $w_{ij}$ denotes the weight for the communication among agents in the network desribed by an undirected graph $\mathcal{G}=(\mathcal{V},\mathcal{E})$. Each node in $\mathcal{V}$ represents an agent, and each edge $\{i,j\} \in \mathcal{E}$ means that $i$ can send messages to $j$ and vice versa. We consider a graph $\mathcal{G}$ satisfying the following assumption.
\begin{ass}\label{graph}
The communication graph $\mathcal{G}$ is undirected and connected, i.e., there exists a path between any two agents.
\end{ass}
We define the mixing matrix $W = \{w_{ij}\}_{1 \leq i,j \leq n}$ as follows. The nonnegative weight $w_{ij}$ is given for each communication link $\{i,j\}\in \mathcal{E},$ where $w_{ij}\neq0$ if $\{i,j\}\in\mathcal{E}$ and $w_{ij} = 0$ if $\{i,j\}\notin\mathcal{E}$. In this paper, we make the following assumption on the mixing matrix $W$. 
\begin{ass}\label{ass-1-1}
The mixing matrix $W = \{w_{ij}\}_{1 \leq i,j \leq n}$ is doubly stochastic, i.e., $W\mathbf{1}=\mathbf{1}$ and $\mathbf{1}^T W = \mathbf{1}^T$. In addition, $w_{ii}>0$ for all $i \in \mathcal{V}$. 
\end{ass}
 
In the following result, we establish that the sequence $\{x_k(t)\}_{k=1}^{n}$ is stable, i.e., uniformly bounded for $t\geq 0$ under suitable conditions.
\begin{thm}\label{thm-1-0}
Suppose that the assumptions 1-2 hold, and assume that one of the following conditions holds:
\begin{enumerate}
\item The following function 
\begin{equation}\label{eq-1-2}
F_{\eta}(x) := \sum_{k=1}^n f_k (x_k) + \frac{1}{2\eta} x^T (I-W) x
\end{equation}
is bounded below and has an optimal solution $(x_1^*, \cdots, x_n^*)$. 
\item Each function $f_k$ is $L$-smooth and $f$ is $\alpha$-stronlgy convex. In addition, the stepsize $\eta>0$ satisfies the following inequality
\begin{equation}\label{eq-1-15}
\eta^2 (\alpha^2 +\alpha L) + \eta \Big( \alpha + L - \frac{(1-\rho_W) \alpha^2}{L}\Big) < \frac{(1-\rho_W) \alpha}{L},
\end{equation}
where $\rho_W$ is the spectral norm of the matrix $W- \frac{1}{n}\mathbf{1}\mathbf{1}^T$.
\end{enumerate}
Then the sequence $\{x_k (t)\}_{k=1}^n$ of the DPPA with stepsize $\eta>0$ is uniformly bounded for $t \geq 0$.
\end{thm}
 We now compare the stability result with that of the DGD described by 
\begin{equation*}
x_i (t+1) = \sum_{j=1}^N w_{ij} x_j (t) - \eta \nabla f_i (x_i (t)).
\end{equation*} 
 Let $\lambda_n (W) \in \mathbb{R}$ be the smallest eigenvalue of $W$. Then it is known from \cite{YLY} that the DGD is stable if the stepsize $\eta>0$ satisfies
\begin{equation*}
\eta \leq \frac{1+\lambda_n (W)}{L},
\end{equation*}
provided that  each cost function $f_j$ is convex and $L$-smooth.  
\begin{table}[ht]\label{tab-5}
\centering
\begin{tabular}{|c|c|c|c| }\cline{1-4}
& Condition on the costs& Stepsize &Paper
\\
\hline
&&&\\[-1.2em]
DGD & \makecell{Each $f_j$ is convex \\ and $L$-smooth} &$\eta \leq \frac{(1+\lambda_n (W))}{L}$   &\cite{YLY}
\\
&&&\\[-1.2em]
\hline
&&&\\[-0.9em]
DPPA &\makecell{$F_{\eta}$ is bounded below \\ and has an optimizer} & $\eta \in (0,\infty)$   & This work
\\[0.9em] 
\hline
\end{tabular}
\vspace{0.1cm}
\caption{This table compares the condition on the stepsize for the stability of  the DGD and  the DPPA.}\label{known results1}
\end{table} 
Table \ref{tab-5} summarizes the condition on the stepsize $\eta>0$ for the stability of the DGD and the DPPA. 
We observe that $F_{\eta}$ defined in \eqref{eq-1-2} is bounded below for any $\eta>0$ when each $f_j$ is convex. 
Therefore, the condition of (1) in Theorem \ref{thm-1-0} is not so restrictive than the condition that $f_j$ is convex which is required for the stability result \cite{YLY} of the DGD. Hence the result of Theorem \ref{thm-1-0} proves that the stepsize choice of $\eta>0$ is much wider for the DPPA than that for the DGD. 

For the convergence analysis of the DPPA, we assume the uniformly boundedness property.  
\begin{ass} The sequence $\{x_k (t)\}_{k=1}^n$ is uniformly bounded for $t \geq 0$, i.e., there is $R>0$ such that
 \begin{equation*}
 A_t \leq R \quad \textrm{and}\quad B_t \leq R\quad \forall ~t \geq 0.
 \end{equation*}
 \end{ass}
 Although the property is guaranteed for broad class of functions by the result of Theorem \ref{thm-1-0}, we formulate this assumption of the sake of simplicity in the statement of the convergence result.
We also consider the following assumption. 
\begin{ass} The aggregate cost function $f$ is $\alpha$-strongly convex for some $\alpha >0$ and each cost function $f_j$ is $L$-smooth for each $1 \leq j \leq n$, i.e., 
\begin{equation*}
\|\nabla f_j (x) - \nabla f_j (y) \| \leq L\|x-y\|
\end{equation*}
for all $x,y\in \mathbb{R}^d$.
\end{ass}
Under this assumption, there exists a 
 unique optimizer $x_* = \arg\min_{x\in\mathbb{R}^d} f(x)$. We let  $D = \max_{1 \leq i \leq n} \|\nabla f_i (x_*)\|$. Also we regard $x_i(t)$ as a row vector in $\mathbb{R}^{1\times d}$, and define the variable  $\mathbf{x}(t)\in \mathbb{R}^{n\times d}$ and $\bar{\mathbf{x}}(t) \in \mathbb{R}^{n \times d}$ by
\begin{equation}\label{eq-2-4}
\mathbf{x}(t)=\left(x_1\left(t\right)^T, \cdots, x_m\left(t\right)^T\right)^T \quad \textrm{and}\quad \bar{\mathbf{x}}(t) = \left(\bar{x}\left(t \right)^T, \cdots, \bar{x}\left( t\right)^T \right)^T,
\end{equation}
where $\bar{x}(t) = \frac{1}{n} \sum_{k=1}^n x_k(t)$.  We let
\begin{equation*} 
A_t = \|\bar{x}(t) -x_*\|  \quad \textrm{and}\quad  B_t = \frac{1}{\sqrt{n}} \|\mathbf{\bar{x}}(t) - \mathbf{x}(t)\|= \Big( \frac{1}{n}\sum_{k=1}^n \|\bar{x}(t) -x_k (t)\|^2 \Big)^{1/2}.
\end{equation*} 
We show that the DPPA converges exponentially to an $O(\eta)$-neighborhood of the optimizer in the following result.
 \begin{thm}\label{thm-1-1}Suppose that the assumptions 1-3 hold and  
Then we have
\begin{equation}\label{eq-1-1}
A_t \leq   \frac{A_0}{(1+\eta\alpha)^{t}} + \frac{t \eta L B_0 \rho_W  }{(1+\eta \alpha)} \max\Big\{\rho_W, \frac{1}{1+\eta \alpha}\Big\}^{t-1} + \frac{\eta L (2RL +D)}{\alpha (1-\rho_W)} 
\end{equation}
and
\begin{equation*}
B_t  \leq (\rho_W)^t B_0 + \frac{\eta}{(1-\rho_W)} (2RL +D).
\end{equation*}
\end{thm}

Upon knowing that  the DGD is stable,  the linear convergence result of the DGD is proved when the stepsize $\eta$ satisfies $\eta \leq \frac{2}{\alpha +L}$. We refer to \cite{YLY, CK} for the detail. As for the DPPA, we note that there is no restriction on the stepsize $\eta>0$ for the linear convergence as obtained in \mbox{Theorem \ref{thm-1-1}.} Table \ref{tab-2} compares the condition on the stepsize of the DGD and the DPPA for the convergence of the algorithms.
\begin{table}[ht]
\centering
\begin{tabular}{|c|c|c|c| c|}\cline{1-5}
& Condition on the costs& DGD & Estimate &Paper
\\
\hline
&&&&\\[-1.2em]
DGD & \makecell{Each $f_j$ is   $L$-smooth \\ $f$ is $\alpha$-strongly convex} &$\eta \leq\frac{2}{\alpha+L}$  & $O(e^{-ct}) + O\Big(\frac{\eta}{1-\rho_W}\Big) $ &\cite{YLY, CK}
\\
&&&&\\[-1.2em]
\hline
&&&&\\[-0.9em]
DPPA &\makecell{Each $f_j$ is  $L$-smooth \\ $f$ is $\alpha$-strongly convex} & $\eta \in (0,\infty)$ &$O(e^{-ct}) +  O\Big(\frac{\eta}{1-\rho_W}\Big)$ & This work
\\[0.9em] 
\hline
\end{tabular}
\vspace{0.1cm}
\caption{This table compares the condition on the stepsize for the stability of the  DGD and the DPPA.}\label{tab-2}
\end{table}

The rest of this paper is organized as follows. In Section 2, we establish two sequential inequalities of $A_t$ and $B_t$. Section 3 is devoted to prove the uniform boundedness result of Theorem \ref{thm-1-0}. In Section 4, we prove the  convergence result of Theorem \ref{thm-1-1}. Numerical results are presented in Section \ref{sec-5}.

\section{Sequential estimates}
 In this section, we dervie two sequential inequalities of $A_t$ and $B_t$, which are main ingredients for the stability and the convergence analysis of the DPPA.

\begin{prop}\label{lem-2-1}Assume that $f$ is $\alpha$-strongly convex. Then, for any stepsize $\eta >0$, the sequence $\{(A_t, B_t)\}_{t \geq 0}$ satisfies the follwoing inequality
\begin{equation}\label{eq-2-12}
(1+\eta \alpha) A_{t+1} \leq A_t + \eta L B_{t+1}
\end{equation} 
for all $t\geq 0$.
\end{prop}
\begin{proof}
From the minimality of \eqref{eq-1-10}, it follows that
\begin{equation}\label{eq-2-3}
\nabla f_k (x_k (t+1)) + \frac{1}{\eta} \Big(x_k (t+1) - \sum_{j=1}^n w_{kj} x_j (t)\Big) =0.
\end{equation} 
We reformulate this as
\begin{equation}\label{eq-2-10}
x_k  (t+1) + \eta \nabla f_k (x_k (t+1))  = \sum_{j=1}^n w_{kj} x_j (t).
\end{equation}
By averaging this for $1 \leq k \leq n$, we get
\begin{equation}\label{eq-2-11}
\bar{x}(t+1) + \frac{\eta}{n} \sum_{k=1}^n \nabla f_k (x_k (t+1)) = \bar{x}(t).
\end{equation}
Using this and the fact that $\nabla f(x_*)=0$, we find
\begin{equation}\label{eq-2-1}
\begin{split}
&\bar{x}(t+1) -x_* + \eta (\nabla f(\bar{x}(t+1)) - \nabla f(x_*))
\\
& = \bar{x}(t) -x_* + \eta \Big( \sum_{k=1}^n \nabla f_k (\bar{x}(t+1)) - \nabla f_k (x_k (t+1))\Big).
\end{split}
\end{equation}
Since $f$ is assumed to be $\alpha$-strongly convex, we have
\begin{equation*}
\|\nabla f(x) - \nabla f(y) \|^2 \geq \alpha^2 \|x-y\|^2
\end{equation*}
and
\begin{equation*}
\langle x -y, ~\nabla f(x) - \nabla f(y) \rangle \geq \alpha \|x-y\|^2.
\end{equation*}
Combining these estimates, we get
\begin{equation*}
\begin{split}
&\Big\|\bar{x}(t+1) -x_* + \eta \Big( \nabla f( \bar{x}(t+1)) - \nabla f(x_*)\Big) \Big\|^2
\\
& = \|\bar{x}(t+1) -x_*\|^2 + 2 \eta \langle \bar{x}(t+1)-x_*,~ \nabla f(\bar{x}(t+1)) - \nabla f(x_*)\rangle 
\\
&\qquad + \eta^2 \|\nabla f(\bar{x}(t+1)) - \nabla f(x_*)\|^2
\\
& \geq (1+2\eta \alpha + \alpha^2) \|\bar{x}(t+1) -x_*\|^2.
\end{split}
\end{equation*}
Using this estimate in \eqref{eq-2-1} and applying the triangle inequality,  we get
\begin{equation*}
\begin{split}
&(1+\eta  \alpha ) \|\bar{x}(t+1) -x_*\|
\\
&\leq   \|\bar{x}(t) -x_*\|  + \frac{\eta}{n} \Big\| \sum_{k=1}^n \nabla f_k (\bar{x}(t+1)) - \nabla f_k (x_k (t+1)) \Big\|
\\
&\leq   \|\bar{x}(t) -x_*\|  +  \frac{\eta L}{n} \sum_{k=1}^n \|\bar{x}(t+1) - x_k (t+1)\| 
\\
& \leq  \|\bar{x}(t) -x_*\|  +  \frac{\eta L}{\sqrt{n}} \|\bar{\mathbf{x}}(t+1) - \mathbf{x}(t+1)\|,
\end{split}
\end{equation*}
where we used the Cauchy-Schwartz inequality in the last inequality. This gives the desired inequality. 
\end{proof}

Next we will derive a bound of $B_{t+1}$ in terms of $A_t$ and $B_t$. For this we will use the following result (see \cite[Lemma 1]{PN}).
\begin{lem}\label{lem-1-1}
Suppose Assumptions \ref{graph} and \ref{ass-1-1} hold, and let $\rho_W$ be the spectral norm of the matrix $W - \frac{1}{n} \mathbf{1} \mathbf{1}^T.$  Then we have $\rho_W<1$ and
\begin{equation*}
\sum_{i=1}^n \Big\| \sum_{j=1}^n w_{ij} (x_j - \bar{x}) \Big\|^2 \leq (\rho_W)^2 \sum_{i=1}^n \|x_i - \bar{x}\|^2,
\end{equation*}
 where $\bar{x} = \frac{1}{n} \sum_{k=1}^n x_k$ 
 for any $x_i\in \mathbb{R}^{d\times1}$ and $1\leq i \leq n$.
\end{lem}
 In the following result, we find an estimate of $B_{t+1}$ in terms of $B_t$ and $A_{t+1}$.
 \begin{prop}\label{lem-2-2}Suppose that each $f_j$ is $L$-smooth. Then the sequence $\{(A_t, B_t)\}_{t \geq 0}$ satisfies the following inequality
 \begin{equation}\label{eq-2-13}
B_{t+1} \leq \rho_W B_t + \eta L B_{t+1} + \eta L A_{t+1} + \eta D
\end{equation}
for all $t \geq 0$.
\end{prop}
 
\begin{proof}
We may write \eqref{eq-2-10} and \eqref{eq-2-11} in the following way
\begin{equation*}
\mathbf{x}(t+1) +\eta \nabla F (\mathbf{x}(t+1)) = W \mathbf{x}(t)
\end{equation*}
and
\begin{equation*}
\bar{\mathbf{x}}(t+1) +\eta \overline{\nabla F }(\mathbf{x}(t+1)) = \bar{\mathbf{x}}(t),
\end{equation*}
where we have let
\begin{equation*}
\overline{\nabla F}(\mathbf{x}(t+1)) = \Big( \nabla f_1 (x_1 (t+1)), \cdots, \nabla f_n (x_n (t+1))\Big)^T.
\end{equation*}
Combining the above equalities, we find 
\begin{equation*} 
\mathbf{x}(t+1) - \bar{\mathbf{x}}(t+1) = W (\mathbf{x}(t)-\bar{\mathbf{x}}(t)) - \eta \Big( \nabla F(\mathbf{x}(t+1)) - \overline{\nabla F}(\mathbf{x}(t+1))\Big).
\end{equation*}
By applying the triangle inequality and Lemma \ref{lem-1-1}, we deduce
\begin{equation}\label{eq-2-2}
\begin{split}
&\|\mathbf{x}(t+1) - \bar{\mathbf{x}}(t+1)\|
\\
&\leq \rho_W \|\mathbf{x}(t) - \bar{\mathbf{x}}(t)\| + \eta \|\nabla F(\mathbf{x}(t+1)) - \overline{\nabla F}(\mathbf{x}(t+1)) \|.
\end{split}
\end{equation}
Using the fact that the spectral radius of the matrix $\Big(I_{n} -\frac{1}{n}1_n^T 1_n\Big)$ is one, we obtain
\begin{equation*}
\begin{split}
&\|\nabla F(\mathbf{x}(t+1)) - \overline{\nabla F}(\mathbf{x}(t+1)) \|
\\
& \leq \| \nabla F(\mathbf{x}(t+1))\|
\\
& \leq  \|\nabla F(\mathbf{x}(t+1)) - \nabla F(\bar{\mathbf{x}}(t+1))\| + \|\nabla F(\bar{\mathbf{x}}(t+1)) -\nabla F(x_*)\| + \|\nabla F(x_*)\|
\\
& \leq L \|\mathbf{x}(t+1) - \bar{\mathbf{x}}(t+1)\| + \sqrt{n} L \|\bar{{x}}(t+1) - {x}_*\| + \sqrt{n} D.
\end{split}
\end{equation*}
Inserting this into \eqref{eq-2-2} we obtain
\begin{equation*}
\begin{split}
&\|\mathbf{x}(t+1) - \bar{\mathbf{x}}(t+1)\|
\\
&\leq \rho_W \|\mathbf{x}(t) - \bar{\mathbf{x}}(t)\| + \eta L \|\mathbf{x}(t+1) - \bar{\mathbf{x}}(t+1)\|
\\
&\quad  + \sqrt{n} L \eta \|\bar{x}(t+1) - x_*\| + \sqrt{n} D\eta,
\end{split}
\end{equation*}
which is the desired inequality. 
\end{proof}
 
 \section{Boundedness of the sequence}
 We prove the uniform boundedness result of Theorem \ref{thm-1-0} under the two conditions separately in the below.
\begin{proof}[Proof of Theorem \ref{thm-1-0}]

 Assume the first condition of Theorem \ref{thm-1-0}. We claim the following inequality 
 \begin{equation}\label{eq-3-1}
\sum_{k=1}^{n} \|x_k (t+1) -x_k^*\| \leq  \sum_{k=1}^n \|x_k (t) -x_k^*\|\quad \forall~ t \geq 0.
\end{equation} 
The optimizer $(x_1^*, \cdots, x_n^*)$ of \eqref{eq-2-2} satisfies
\begin{equation*}
\nabla f_k (x_k^*) + \frac{1}{\eta} \Big( x_k^* -\sum_{j=1}^n w_{kj} x_j^* \Big) = 0.
\end{equation*}
Combining this with \eqref{eq-2-3} gives
\begin{equation*}
\eta( \nabla f_k (x_k (t+1))- \nabla f_k (x_k^*)) +   (x_k (t+1) -x_k^*) = \sum_{j=1}^n w_{kj}( x_j (t)-x_j^*).
\end{equation*}
From this we have
\begin{equation*}
\|x_k (t+1) -x_k^*\| \leq \Big\|\sum_{j=1}^n w_{kj} (x_j (t) -x_j^*)\Big\| \leq \sum_{j=1}^n w_{kj} \|x_j (t) -x_j^*\|.
\end{equation*}
Summing up this for $1\leq  k \leq n$, we get
\begin{equation*}
\sum_{k=1}^{n} \|x_k (t+1) -x_k^*\| \leq \sum_{k=1}^n  \sum_{j=1}^n w_{kj} \|x_j (t) -x_j^*\| = \sum_{j=1}^n \|x_j (t) -x_j^*\|,
\end{equation*}
which proves the inequality \eqref{eq-3-1}. It gives the following bound
\begin{equation*}
\sum_{k=1}^{n} \|x_k (t) -x_k^*\| \leq \sum_{k=1}^n  \sum_{j=1}^n w_{kj} \|x_j (t) -x_j^*\| = \sum_{j=1}^n \|x_j (0) -x_j^*\|\quad \forall~ t\geq 0.
\end{equation*}
Hence $A_t$ and $B_t$ are uniformly bounded.

Next we assume the second condition of Theorem \ref{thm-1-0}.  
Then we claim that the sequence $A_t$ and $B_t$ satisfies
\begin{equation*}
A_t \leq R \quad \textrm{and}\quad B_t \leq \frac{\alpha}{L} R,
\end{equation*}
where $R>0$ is defined by 
\begin{equation}\label{eq-3-3}
R = \max \Bigg\{ A_0, ~\frac{L}{\alpha}B_0, ~\frac{L}{\alpha}B_1,~ \frac{\eta D}{\frac{\alpha}{L} \Big( 1- \eta L - \frac{\eta^2 L^2}{(1+\eta \alpha)} \Big) - \frac{\alpha}{L} \rho_W - \frac{\eta L}{(1+\eta \alpha)}}\Bigg\}.
\end{equation}
We argue by induction to prove this claim. First, we note that $A_0 \leq R$, $B_0 \leq R$ and $B_1 \leq R$ by the definition of $R$. Next we assume that
\begin{equation}\label{eq-3-10}
A_t \leq R \quad \textrm{and}\quad B_{t+1} \leq cR 
\end{equation}
for some $t \geq 0$ and $c = \frac{\alpha}{L}$. Then, we use these bounds in  \eqref{eq-2-12}  to find
\begin{equation*}
(1+\eta \alpha)A_{t+1} \leq (R+\eta Lc R) = (1+\eta \alpha)R,
\end{equation*}
which gives $A_{t+1} \leq R$. Next we recall the estimates \eqref{eq-2-12} and \eqref{eq-2-13} with $t$ replaced by $(t+1)$ as
\begin{equation*}
B_{t+2} \leq \rho_W B_{t+1} + \eta L B_{t+2} + \eta L A_{t+2} + \eta D
\end{equation*}
and
\begin{equation*}
\begin{split}
(1+\eta \alpha) A_{t+2}& \leq A_{t+1} + \eta L B_{t+2}
\\
&\leq R + \eta L B_{t+1},
\end{split}
\end{equation*}
where we used $A_{t+1} \leq R$ in the last inequality.
Combining these estimates with \eqref{eq-3-10} gives
\begin{equation*}
\Big( 1- \eta L -  \frac{\eta^2 L^2}{(1+\eta \alpha)} \Big) B_{t+2} \leq (\rho_W ) (cR) + \frac{\eta L}{1+\eta \alpha} R + \eta D.
\end{equation*}
This gives $B_{t+2} \leq cR$ provided that 
\begin{equation*}
 \rho_W (cR) + \frac{\eta L}{(1+\eta \alpha)} R + \eta D \leq Rc \Big( 1- \eta L -  \frac{\eta^2 L^2}{(1+\eta \alpha)} \Big).
\end{equation*}
This holds true for $R>0$ defined in \eqref{eq-3-3} and $\eta>0$ satisfying
\begin{equation*}
\frac{\alpha}{L} \Big((1-\rho_W) - \eta L - \frac{\eta^2 L^2}{(1+\eta \alpha)}\Big) > \frac{\eta L}{(1+\eta \alpha)},
\end{equation*}
which is same with 
\begin{equation*}
\eta^2 (\alpha^2 +\alpha L) + \eta \Big( \alpha + L - \frac{(1-\rho_W) \alpha^2}{L}\Big) < \frac{(1-\rho_W) \alpha}{L}.
\end{equation*}
The proof is done.

\end{proof}

\section{Convergence result}

In this section, we prove the main convergence result of the decentralized proximal point method. 
 
\begin{proof}[Proof of Theorem \ref{thm-1-1}]
By Proposition \ref{lem-2-1} and Proposition \ref{lem-2-2}  we have the following inequalities
\begin{equation}\label{eq-4-1}
(1+\eta \alpha) A_{t+1} \leq A_t + \eta L B_{t+1}
\end{equation}
and
\begin{equation*}
B_{t+1} \leq \rho_W B_t + \eta L B_{t+1} + \eta L A_{t+1} + \eta D.
\end{equation*}
Using the assumption that $A_{t+1} \leq R$ and $B_{t+1} \leq R$, we have
\begin{equation*}
B_{t+1} \leq (\rho_W ) B_t + \eta (2RL +D)
\end{equation*}
for all $t \geq 0$. Using this iteratively, we get
\begin{equation*}
\begin{split}
B_t & \leq (\rho_W)^t B_0 + \eta (2RL +D) \Big[ 1 + \rho_W + \cdots + (\rho_W)^{t-1}\Big]
\\
& \leq (\rho_W)^t B_0 + \frac{\eta}{(1-\rho_W)} (2RL +D).
\end{split}
\end{equation*}
Putting this inequality into \eqref{eq-4-1} leads to
\begin{equation}\label{eq-4-2}
\begin{split}
&(1+\eta \alpha)A_{t+1}
\\
& \leq A_t + \eta L \Big[ (\rho_W)^{t+1} B_0 + \frac{\eta}{(1-\rho_W)} (2RL +D)\Big] 
\\
& = A_t + \eta L B_0 (\rho_W)^{t+1} + \frac{\eta^2 L}{(1-\rho_W)} (2RL +D).
\end{split}
\end{equation}
To analyize this sequential estimate, we consider two positive sequences $\{z_t\}_{t \geq 0}$ and $\{y_t\}_{t \geq 0}$ satisfying
\begin{equation*}
(1+\eta\alpha) z_{t+1} = z_t + \eta L B_0 (\rho_W)^{t+1} 
\end{equation*}
and 
\begin{equation*}
(1+\eta \alpha) y_{t+1}= y_t + \frac{\eta^2 L}{(1-\rho_W)} (2RL +D),
\end{equation*}
where $z_0 = A_0$ and $y_0 = 0$. It then follows from \eqref{eq-4-2} that $A_t \leq z_t +y_t$ for all $t \geq 0$.

We estimate $z_{t }$ as follows:
\begin{equation*}
\begin{split}
z_t & = \frac{z_0}{(1-\eta\alpha)^{t}} + \frac{\eta L B_0 \rho_W}{(1+\eta \alpha)} \Big[ \sum_{j=0}^{t-1} \frac{(\rho_W)^j}{(1+\eta \alpha)^{t-1-j}}\Big]
\\
& \leq \frac{z_0}{(1-\eta\alpha)^{t}} + \frac{t \eta L B_0\rho_W   }{(1+\eta \alpha)} \max\Big\{\rho_W, \frac{1}{1+\eta \alpha}\Big\}^{t-1}.
\end{split}
\end{equation*}
Next we estimate $y_t$ as
\begin{equation*}
\begin{split}
y_t & = \frac{1}{(1+\eta\alpha)^t} y_0 + \frac{\eta^2 L}{(1+\eta \alpha) (1-\rho_W)} (2RL +D) \sum_{j=0}^{t-1}  \frac{1}{(1+\eta \alpha)^{j-1}}
\\
& \leq \frac{1}{(1+\eta\alpha)^t} y_0 + \frac{\eta^2 L}{(1+\eta \alpha) (1-\rho_W)} (2RL +D) \sum_{j=0}^{\infty} \frac{1}{(1+\eta \alpha)^{j}}
\\
&=\frac{y_0}{(1+\eta \alpha)^t} + \frac{\eta L (2RL +D)}{\alpha (1-\rho_W)}.
\end{split}
\end{equation*}
Combining the above estimates gives
\begin{equation*}
\begin{split}
A_t & \leq z_t + y_t
\\
 &\leq   \frac{A_0}{(1-\eta\alpha)^{t}} + \frac{t \eta L B_0 \rho_W }{(1+\eta \alpha)} \max\Big\{\rho_W, \frac{1}{1+\eta \alpha}\Big\}^{t-1} + \frac{\eta L (2RL +D)}{\alpha (1-\rho_W)}.
 \end{split}
\end{equation*}
This corresponds to the inequality \eqref{eq-1-15} in the condition (2).  Thus, it holds that $A_{t+1} \leq R$ and $B_{t+2} \leq cR$. This completes the induction, and so the claim is proved. Hence the uniform boundedness is proved. 
\end{proof}

\section{Numerical tests}\label{sec-5}

This section gives numerical experimental results of the DPPA. We consider the cost function
\begin{equation*}
f(x)  = \frac{1}{n} \sum_{k=1}^n \|A_k x- y_k\|^2,
\end{equation*}
where $n$ is the number of agents and for each $1 \leq i \leq n$ and $A_i$ be an $m\times d$ matrix whose element is chosen randomly following the normal distribution $N(0,1)$. Also we set $y_i \in \mathbb{R}^m$ whose each element  is generated from the normal distribution $N(0,1)$. We choose $n=20$, $m=5$ and $d=10$.

For the communication matrix $W$, we link each two agents with probability 0.4, and define the weights $w_{ij}$ by 
\begin{equation*}
w_{ij}= \left\{ \begin{array}{ll} 1/ \max\{\textrm{deg}(i), \textrm{deg}(j)\}&~\textrm{if}~ i \in N_i,
\\
1- \sum_{j \in N_i} w_{ij}&~\textrm{if}~ i=j,
\\
0& ~otherwise.
\end{array}
\right.
\end{equation*} 
We define the following value
\begin{equation*}
\eta_c = \frac{1+\lambda_n (W)}{L},
\end{equation*}
where $\lambda_n (W)$ is the smallest eigenvalue of $W$ and $L$ is the constant for the smoothness propery of $f_j$ for $1 \leq j \leq n$. In  our experiment, the constants are computed as $L\simeq 29.7312$ and $\lambda_n (W) \simeq -0.4009$. Therefore we find $\eta_c \simeq 0.0202$.

We perform the numerical simulation for the DPPA and the DGD with stepsize chosen as $\eta = \eta_c  +0.005$ and $\eta = \eta_c$. Figure \ref{fig1} indicates the graphs of 
 the error $\log(\sum_{k=1}^{n}\|x_k (t) -x_*\|/n)$ with respect to $t \geq 0$. The result shows that both the DPPA and the DGD are stable for $\eta =\eta_c$ as the theoretical results of Table \ref{known results1} guarantee. On the other hand, if we choose $\eta = \eta_c + 0.005$, then the DGD becomes unstable while the DPPA is still stable. This supports  the result of Theorem \ref{thm-1-0}.  
\begin{figure}[htbp]
\includegraphics[height=5.5cm, width=7.3cm]{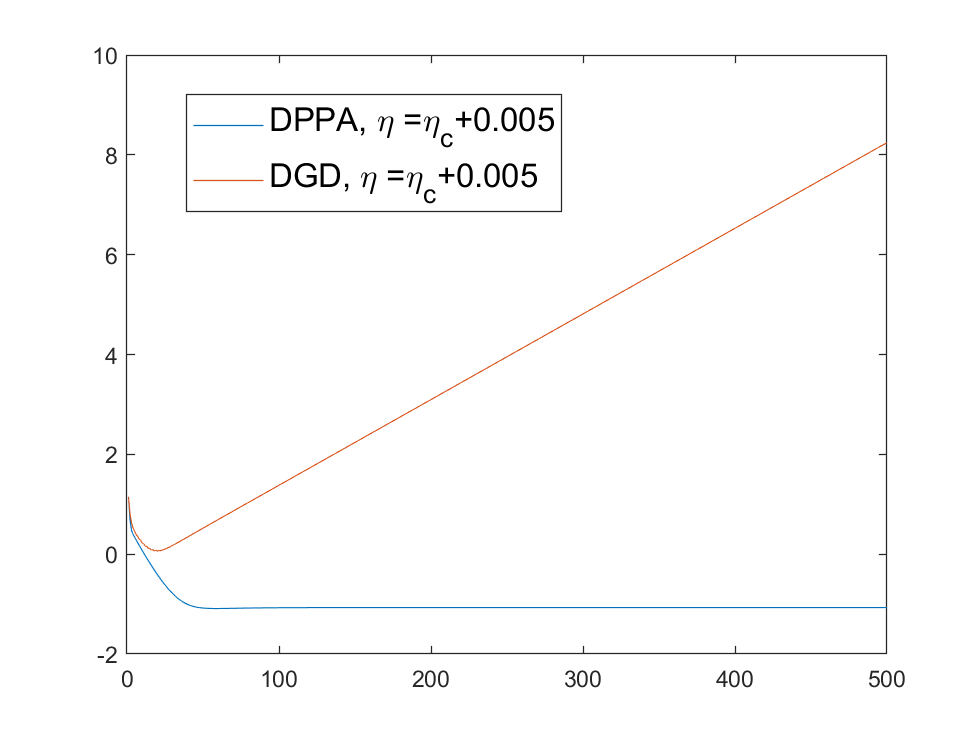} \includegraphics[height=5.5cm, width=7.3cm]{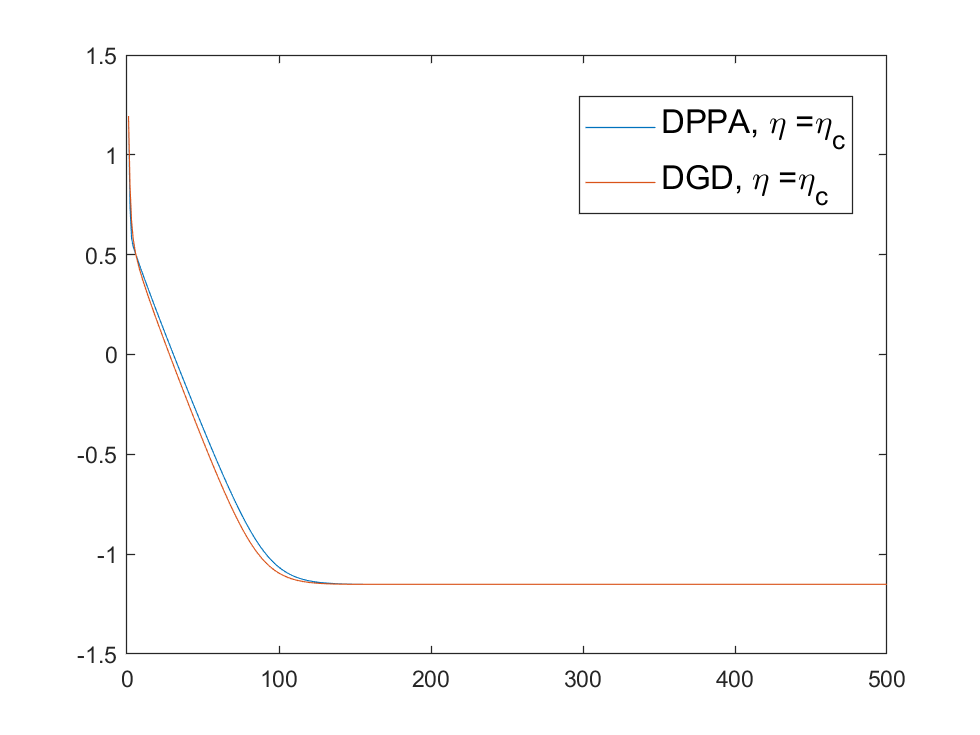}
\vspace{-0.3cm}\caption{The graphs of the error $\log(\sum_{k=1}^{n}\|x_k (t) -x_*\|/n)$ for the DPPA and the DGD for $t \geq 0$ with stepsize $\eta =\eta_c + 0.005$ and $\eta =\eta_c$.}\label{fig1}
\end{figure}
Next we perform the numerical test for the DPPA with various stepsize $\eta \in \{0.001, 0.01, 0.1, 1,2\}$. We measure the error $\log(\sum_{k=1}^{n}\|x_k (t) -x_*\|/n)$  and the graphs are given in Figure \ref{fig2}. The result shows that the error decreases exponentially up to an $O(\eta)$ value, which supports the convergence result of Theorem \ref{thm-1-1}.
\begin{figure}[htbp]
\includegraphics[height=5.5cm, width=7.3cm]{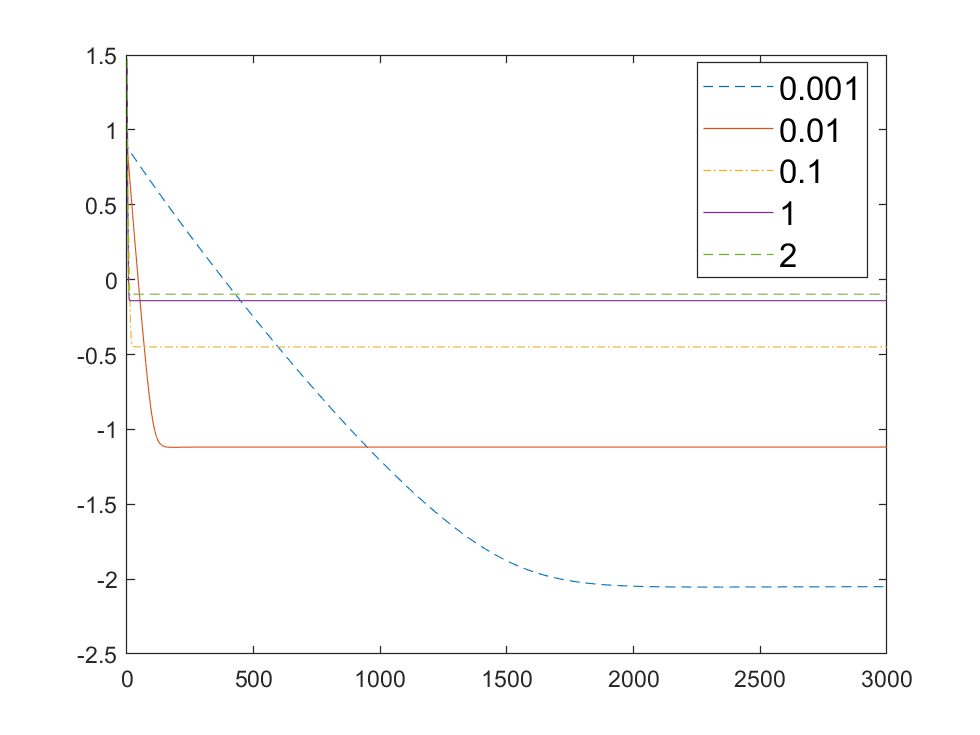} 
\vspace{-0.3cm}\caption{The graphs of the error $\log(\sum_{k=1}^{n}\|x_k (t) -x_*\|/n)$ for the DPPA  with respect to  $t \geq 0$ and stepsize $\eta\in \{0.001, 0.01, 0.1, 1, 2\}$.}\label{fig2}
\end{figure}

\section*{Acknowledgments}
The author was supported by the National Research Foundation of Korea (NRF) grants funded by the Korea government No. NRF-2016R1A5A1008055 and No. NRF-2021R1F1A1059671.

\end{document}